\documentclass[11pt,a4paper]{article}

\usepackage{amsmath}
\usepackage{amssymb,latexsym}
\usepackage{mathabx}
\usepackage{amsthm}
\usepackage[shortlabels]{enumitem}
\usepackage{array}
\usepackage[hang,flushmargin,bottom]{footmisc}
\usepackage[margin=0.35in,format=hang,font=small,labelfont=bf]{caption}
\usepackage[normalem]{ulem}
\usepackage{needspace}
\usepackage{graphicx,tikz}
\usepackage[colorlinks=true,urlcolor=blue,linkcolor=blue,citecolor=blue]{hyperref}
\usepackage{palatino}
\usepackage{eulervm}

\usepackage{ifthen}
\usetikzlibrary{calc}
\usetikzlibrary{decorations.pathreplacing,decorations.markings}

\newcommand{\plotptradius}{4pt}

\tikzset{permpt/.style={circle, draw, fill=black, inner sep=0pt, minimum width=\plotptradius}}
\tikzset{empty/.style={draw=none, fill=none}}

\newcommand\absdot[2]{
	\node[permpt] at #1 {};
}

\newcommand{\plotperm}[2][black]{ %[colour]{permutation}
	\foreach \j [count=\i] in {#2} {
        \ifnum0=\j {} \else {
 		\node[permpt,fill=#1,draw=#1] (\j) at (\i,\j) {};
	} \fi
	}
}
\newcommand{\plotpermbox}[4]{
	\draw [darkgray, very thick, line cap=round, fill=white]
		({#1-0.5}, {#2-0.5}) rectangle ({#3+0.5}, {#4+0.5});
}

\newcommand{\plotpermborder}[1]{
	\foreach \i [count=\nn] in {#1} {\global\let\n\nn};    % scope of \nn now restricted to loop, resolving tikz issue #702, 3 Jul 2019
	\plotpermbox{1}{1}{\n}{\n}
	\plotperm{#1}
}
\newcommand{\plotpermbordergrid}[1]{
	\foreach \i [count=\nn] in {#1} {\global\let\n\nn};    % scope of \nn now restricted to loop, resolving tikz issue #702, 3 Jul 2019
	\plotpermbox{1}{1}{\n}{\n}
	\draw[step=1cm,gray!50,very thin] (1,1) grid (\n,\n);
	\plotperm{#1}
}
\newcommand{\plotgrid}[1]{
	\draw[step=1cm,gray!50,very thin] (0.5,0.5) grid (#1+0.5,#1+0.5);
}

\newcommand{\plotpermgrid}[1]{
	\foreach \i [count=\nn] in {#1} {\global\let\n\nn};    % scope of \nn now restricted to loop, resolving tikz issue #702, 3 Jul 2019
	\plotgrid{\n}
	\plotperm{#1}
}

\newcommand{\plotpinsequence}[1]{
	\absdot{(0,0)}{};
	\edef\n{0}
	\edef\s{0}
	\edef\e{0}
	\edef\w{0}
	\edef\x{0}
	\edef\y{0}
	\foreach \pin [remember=\pin as \oldpin (initially 1), count=\i] in {#1} {
		\ifthenelse{\pin=1 \OR \pin=2}{%up or down
			\ifthenelse{\oldpin=3}{% previous=right
				\xdef\x{\number\numexpr\e-1}
			}{
				\xdef\x{\number\numexpr\w+1}
			}
			\ifnum\i=1 %expand eastern box by 1 if 1st pin
				\pgfmathparse{\e+1}
 				\xdef\e{\pgfmathresult}
			\fi	
		}{ %left or right
			\ifthenelse{\oldpin=1}{% previous=up
				\xdef\y{\number\numexpr\n-1}
			}{
				\xdef\y{\number\numexpr\s+1}
			}
			\ifnum\i=1 %expand southern boundary by 1 if 1st pin
				\pgfmathparse{\s-1}
 				\xdef\s{\pgfmathresult}
			\fi	
		}
		\ifnum\pin=1 %up
			\pgfmathparse{\n+2}
 			\xdef\n{\pgfmathresult}		
			\absdot{(\x,\n)}{};
			\ifnum\i>1
				\draw (\x,\n) -- (\x,\y-0.5);
			\else
				\draw[gray,very thick] (-0.5,-0.5) rectangle (\x+0.5,\n+0.5);
			\fi
		\fi
		\ifnum\pin=2 % down		
			\pgfmathparse{\s-2}
 			\xdef\s{\pgfmathresult}
			\absdot{(\x,\s)}{};
			\ifnum\i>1
				\draw (\x,\s) -- (\x,\y+0.5);
			\else
				\draw[gray,very thick] (-0.5,0.5) rectangle (\x+0.5,\s-0.5);
			\fi
		\fi
		\ifnum\pin=3 %right
			\pgfmathparse{\e+2}
 			\xdef\e{\pgfmathresult}
			\absdot{(\e,\y)}{};
			\ifnum\i>1
				\draw (\e,\y) -- (\x-0.5,\y);
			\else
				\draw[gray,very thick] (-0.5,+0.5) rectangle (\e+0.5,\y-0.5);
			\fi
		\fi
		\ifnum\pin=4 %left
			\pgfmathparse{\w-2}
 			\xdef\w{\pgfmathresult}
			\absdot{(\w,\y)}{};
			\ifnum\i>1
				\draw (\w,\y) -- (\x+0.5,\y);
			\else
				\draw[gray,very thick] (0.5,0.5) rectangle (\w-0.5,\y-0.5);

			\fi
		\fi		
	};
}

\tikzset{
  on each segment/.style={
    decorate,
    decoration={
      show path construction,
      moveto code={},
      lineto code={
        \path [#1]
        (\tikzinputsegmentfirst) -- (\tikzinputsegmentlast);
      },
      curveto code={
        \path [#1] (\tikzinputsegmentfirst)
        .. controls
        (\tikzinputsegmentsupporta) and (\tikzinputsegmentsupportb)
        ..
        (\tikzinputsegmentlast);
      },
      closepath code={
        \path [#1]
        (\tikzinputsegmentfirst) -- (\tikzinputsegmentlast);
      },
    },
  },
  mid arrow/.style={postaction={decorate,decoration={
        markings,
        mark=at position .5 with {\arrow[#1]{stealth}}
      }}},
}

\newtheorem{thm}{Theorem}[section]
\newtheorem*{thm*}{Theorem}
\newtheorem{prop}[thm]{Proposition}
\newtheorem{cor}[thm]{Corollary}
\newtheorem{lemma}[thm]{Lemma}

\newtheorem*{conj*}{Conjecture}

\theoremstyle{definition}

\newtheorem{defn*}{Definition}

\newtheorem*{example*}{Example}

\newtheorem*{comment*}{Comment}

\let\C\CCC
\newcommand{\DDD}{\mathcal{D}}
\newcommand{\EEE}{\mathcal{E}}

\newcommand{\MMM}{\mathcal{M}}

\newcommand{\fS}{\mathfrak{S}}
\newcommand{\fP}{\mathfrak{P}}
\newcommand{\av}{\mathsf{Av}}
\newcommand{\gr}{\mathrm{gr}}

\newcommand{\Grid}{\mathsf{Grid}}

\newcommand{\juxt}[2]{#1\,#2}
\newcommand{\juxtbar}[2]{#1\,|\,#2}

\newcommand{\gridded}{\sharp}

\newcommand{\peaks}{\mathsf{Peaks}}

\tikzstyle{vertex}=[circle, draw, fill=black,
                        inner sep=0pt, minimum width=4pt]

\title{Labelled well-quasi-order in juxtapositions of permutation classes}

\author{Robert Brignall\\
\small School of Mathematics and Statistics\\
\small The Open University, UK
}

\begin{document}
\maketitle

\begin{abstract}
The \emph{juxtaposition} of permutation classes $\C$ and $\DDD$ is the class of all permutations formed by concatenations $\sigma\tau$, such that $\sigma$ is order isomorphic to a permutation in $\C$, and $\tau$ to a permutation in~$\DDD$.

We give simple necessary and sufficient conditions on the classes $\C$ and $\DDD$ for their juxtaposition to be labelled well-quasi-ordered (lwqo): namely that both $\C$ and $\DDD$ must themselves be lwqo, and at most one of $\C$ or $\DDD$ can contain arbitrarily long zigzag permutations. We also show that every class without long zigzag permutations has a growth rate which must be integral.
\end{abstract}

% These L-shaped % symbols are just to make it easier 
% to navigate the TeX source and find the starts of sections!
%
%
%
%
%
%
%
%
%%%%%%%%%%%%%%%%%%%%%%
{\centering \emph{For Sophie}\par}

\section{Introduction}

Let $\C$ and $\DDD$ be permutation classes. The \emph{juxtaposition} $\juxt{\C}{\DDD}$ is the permutation class comprising all permutations formed by concatenations $\sigma\tau$, where $\sigma$ is order isomorphic to a permutation in $\C$ and $\tau$ is order isomorphic to a permutation in $\DDD$.

A \emph{zigzag permutation} (or just \emph{zigzag}) is a permutation $\pi=\pi(1)\cdots\pi(n)$ with the property that there is no index $i\in[n-2]$ such that $\pi(i)\pi(i+1)\pi(i+2)$ forms a monotone increasing or decreasing pattern.\footnote{Zigzag permutations (sometimes called \emph{alternating permutations}, but we reserve the term `alternating' for other purposes) have been widely studied in relation to enumerative problems, and are strongly related to the Euler numbers (sequence A000111 of the OEIS~\cite{oeis}) -- for a survey, see Stanley~\cite{stanley:survey-alternating:}.}
The main purpose of this note is to establish the following theorem.

\begin{thm}\label{thm:main}
The juxtaposition $\juxt{\C}{\DDD}$ is labelled-well-quasi-ordered if and only if both $\C$ and $\DDD$ are lwqo, and at least one of $\C$ or $\DDD$ contains only finitely many zigzag permutations.
\end{thm}

The juxtaposition of permutations was first introduced in Atkinson's foundational work~\cite{atkinson:restricted-perm:}, and has since been studied in terms of enumeration (see, for example,~\cite{bs:context-free}) since it represents a natural yet non-trivial way to combine two permutation classes. Indeed, juxtapositions are a special case of \emph{grid classes}, which we define in the next section.

The study of well-quasi-ordering and infinite antichains in permutation classes dates back to the 1970s in the work of Tarjan~\cite{tarjan:sorting-using-n:} and Pratt~\cite{pratt:computing-permu:}, and rose to prominence in the 2000s as a result of works such as Atkinson, Murphy and Ru\v{s}kuc~\cite{atkinson:partially-well-:} and Murphy and Vatter~\cite{murphy:profile-classes:}. The stronger notion of \emph{labelled} well-quasi-ordering dates back to Pouzet~\cite{pouzet:un-bel-ordre-da:}, but received little attention in the context of permutation classes until the current author's recent work with Vatter~\cite{bv:lwqo-for-pp:}. 

The rest of this paper is organised as follows. In Section~\ref{sec:prelim} we briefly cover the requisite terminology. In Section~\ref{sec:zigzags} we provide a necessary and sufficient characterisation of permutation classes without long zigzags. As a by-product of this characterisation, we show that every permutation class without long zigzags has an integral growth rate. In Section~\ref{sec:lwqo} we prove that the juxtaposition of a labelled well-quasi-ordered permutation class with $\av(21)$ or $\av(12)$ is again labelled well-quasi-ordered, and this, together with the characterisation from Section~\ref{sec:zigzags}, enables us to complete our proof of Theorem~\ref{thm:main}. We finish with some concluding remarks in Section~\ref{sec:conclusion}.

\section{Preliminaries}\label{sec:prelim}

\paragraph{Permutation classes}
We provide here only the minimum terminology required for our purposes, and refer the reader to~\cite{bevan2015defs} for fuller details.

A permutation of length $n$, typically denoted $\pi=\pi(1)\cdots\pi(n)$, is an ordering of the symbols in $[n]=\{1,\dots,n\}$. We say that $\sigma=\sigma(1)\cdots\sigma(k)$ is \emph{contained} in $\pi$, and write $\sigma\leq \pi$, if there exists a subsequence $1\leq i_1\leq\cdots\leq i_k\leq n$ such that the relative ordering of the points in $\pi(i_1)\cdots\pi(i_k)$ is the same as that of $\sigma$. That is, $\pi$ contains a subsequence that is \emph{order isomorphic} to $\sigma$.

A \emph{permutation class} $\C$ is a set of permutations closed downwards under containment. Every such class can be described by its set of minimal forbidden elements, but for our purposes it suffices to record that $\av(21)=\{1,12,123,\dots\}$ is the class of increasing permutations, and $\av(21)=\{1,21,321,\dots\}$ is the class of decreasing permutations. Two other classes we will require are as follows
\begin{gather*}
\bigoplus 21 = \{ \text{finite subpermutations of }21436587\cdots\} = av(231,312,321),\\
\bigominus 21 = \{ \text{finite subpermutations of }\cdots78563412\} = av(123,132,213).
\end{gather*}

One important family of permutation classes in the structural study of permutations are \emph{grid classes}. These are defined by a \emph{gridding matrix} $\MMM$ of permutation classes, and each permutation in $\Grid(\MMM)$ has the property that its plot can be divided using horizontal and vertical lines into a grid of cells, of the same dimensions as $\MMM$, and such that the entries in each cell of the plot are order isomorphic to a permutation that belongs to a class in the corresponding cell of $\MMM$. 

Of particular note are \emph{monotone} grid classes, where each cell of $\MMM$ is $\av(21)$, $\av(12)$ or empty, and we say that a permutation class $\C$ is \emph{monotone griddable} if it is the subclass of some monotone grid class. We need the following characterisation.

\begin{thm}[{Huczynska and Vatter~\cite[Theorem 2.5]{Huczynska2006}}]\label{thm:hv}
	A permutation class is monotone griddable if and only if it has finite intersection with $\bigoplus21$ and $\bigominus 12$.
\end{thm}

The juxtaposition $\juxt\C\DDD$ can alternatively be considered as $\Grid(\MMM)$ where $\MMM=\begin{bmatrix}\C&\DDD\end{bmatrix}$. Also of interest to us is the class of \emph{gridded} permutations in a juxtaposition -- denoted $\juxtbar\C\DDD$ -- whose members comprise the permutations of $\juxt\C\DDD$ together with a vertical line that witnesses the permutation's membership of the juxtaposition. Note that each permutation in $\juxt\C\DDD$ can correspond to more than one gridded permutation in $\juxtbar\C\DDD$. The same notion exists for grid classes defined by larger matrices: if $\C\subseteq \Grid(\MMM)$ then $\C^\gridded$ denotes the set of permutations in $\C$ equipped with horizontal and vertical lines to witness their membership of $\Grid(\MMM)$.

\paragraph{Well-quasi-ordering}
A quasi-order $(P,\leq)$ is \emph{well-quasi-ordered} (wqo) if it contains no infinite descending chain, and no infinite antichain -- that is, a set of pairwise incomparable elements. For quasi-ordered classes of combinatorial objects (such as permutation classes or gridded permutation classes), this condition typically reduces to checking for the presence of infinite antichains. 

Given a quasi-order $(P,\leq)$, let $P^*$ denote the set of finite sequences of $P$. The set $P^*$ can be ordered using the \emph{generalised subword order}: for $v=v_1\cdots v_m$ and $w=w_1\cdots w_n$ in $P^*$, we say that $v\preceq w$ if there exists a subsequence $1\leq i_1\leq\cdots\leq i_m\leq n$ such that $v_j\leq w_{i_j}$ for all $1\leq j\leq m$. One celebrated result that we will need is Higman's lemma:

\begin{lemma}[Higman~\cite{higman:ordering-by-div:}]\label{lem:higman}
If $(P,\leq)$ is a wqo set, then so is $(P^*,\preceq)$.
\end{lemma}
 
Another way to combine wqo sets and obtain another wqo set is by taking products:

\begin{prop}[{See~\cite[Proposition 1.2]{bv:lwqo-for-pp:}}]\label{prop:product}
	Let $(P,\leq_P)$ and $(Q,\leq_Q)$ be wqo sets. Then $P\times Q$ is wqo under the product order, $(p_1,q_1)\leq (p_2,q_2)$ if and only if $p_1\leq_P p_2$ and $q_1\leq_Q q_2$.
\end{prop}

The final piece of core wqo machinery we require is as follows. We say that a mapping $\Phi:P\to Q$ between two quasi-orders is \emph{order preserving} if $p_1\leq_P p_2$ implies $\Phi(p_1)\leq_Q \Phi(p_2)$. We have:

\begin{prop}[{See~\cite[Proposition 1.10]{bv:lwqo-for-pp:}}]\label{prop-order-preserving}
Let $(P,\leq_P)$ and $(Q,\leq_Q)$ be quasi-orders, and suppose that $\Phi:	(P,\leq_P)\to (Q,\leq_Q)$ is an order-preserving surjection. If $(P,\leq_P)$ is wqo, then so is $(Q,\leq_Q)$.
\end{prop}

\paragraph{Labelled well-quasi-ordering}
Let $(L,\leq_L)$ be any quasi-order. An $L$-labelling of a permutation $\pi$ of length $n$ (or of a gridded permutation of length $n$) is a mapping $\ell_\pi$ from the indices of $\pi$ to elements of $L$. We write the resulting $L$-labelled permutation as $(\pi,\ell_\pi)$, and the set of all $L$-labelled permutations from some set (or class) $\C$ is denoted $\C\wr L$. 

The set $\C\wr L$ induces a natural ordering: Let $\sigma,\pi\in\C$ be of lengths $m$ and $n$, respectively. We say that $(\sigma,\ell_\sigma)$ is contained in $(\pi,\ell_\pi)$ if there exists a subsequence $1\leq i_1\leq\cdots\leq i_m\leq n$ such that $\pi(i_1)\cdots\pi(i_m)$ is order isomorphic to $\sigma$, and $\ell_\sigma(j) \leq_L \ell_\pi(i_j)$ for all $j\in[m]$.

Finally, a set or class $\C$ is \emph{labelled well-quasi-ordered} (lwqo) if $\C\wr L$ is a wqo set for every wqo set $(L,\leq_L)$. We refer the reader to~\cite{bv:lwqo-for-pp:} for a complete treatment of lwqo in permutation classes.

\section{Zigzags}\label{sec:zigzags}

A \emph{peak} of a permutation $\pi$ is a position $i$ such that $\pi(i-1)<\pi(i)>\pi(i+1)$. The \emph{peak set} of $\pi$ is \[\peaks(\pi)=\{ i: \pi(i-1)<\pi(i)>\pi(i+1)\}.\]
The peak set has been much studied in enumerative and algebraic combinatorics, see, for example, Nyman~\cite{nyman:the-peak:} and Billey, Burdzy and Sagan~\cite{billey:permutations-with:}, although here it simply provides convenient terminology to prove the following result.

\begin{prop}\label{prop:w-class}
Let $\C$ be a permutation class that contains only finitely many zigzags. Then $\C$ is contained in $\Grid(\MMM)$ for a matrix $\MMM$ comprising one row, and in which each entry is $\av(21)$ or $\av(12)$.\footnote{The grid classes appearing in Proposition~\ref{prop:w-class}  were originally introduced by Atkinson, Murphy and Ru\v{s}kuc~\cite{atkinson:partially-well-:} under the term `$W$-classes'.}
\end{prop}

\begin{proof}
Suppose that the longest zigzag in $\C$ has length $k$. For any $\pi\in\C$ of length $n$ consider the peak set $\peaks(\pi)$ and let $i$ and $j$ be two consecutive peaks (that is, there is no $k\in\peaks(\pi)$ such that $i<k<j$). Since there are no peaks between $i$ and $j$, the sequence $\pi(i)\cdots \pi(j)$ must be a \emph{valley}: that is, it is formed of a decreasing sequence, followed by an increasing sequence. Let $v_i$ be the index such that $i<v_i<j$ for which $\pi(v_i)$ is minimal (the `bottom of the valley'). Similarly, if $\ell$ is the leftmost peak in $\pi$, then $\pi(1)\cdots\pi(\ell)$ is a valley, and if $r$ is the rightmost peak in $\pi$, then $\pi(r)\cdots\pi(n)$ is a valley. In particular, we set $v_r$ to be the index in $[r,n]$ for which $\pi(v_r)$ is minimal.

Since the entries between consecutive peaks (and before the first, and after the last peak) form valleys, we see that the entries of $\pi$ can be partitioned into a sequence of $2(|\peaks(\pi)|+1)$ (possibly empty) intervals of entries, that alternately form decreasing and increasing permutations. 
By construction, the subpermutation formed on the indices $\peaks(\pi) \cup \{v_i:i\in\peaks(\pi)\}$ is a zigzag of length $2|\peaks(\pi)|$. Thus, $2|\peaks(\pi)| \leq k$ for every $\pi\in\C$, and hence $\pi$ belongs to the grid class whose matrix is
\[\begin{bmatrix}\av(12)&\av(21)&\av(12)&\av(21)&\cdots&\av(12)&\av(21)\end{bmatrix}\]
comprising $k+2$ cells (if $k$ is even), or $k+1$ cells (if $k$ is odd). 
\end{proof}

Our next result establishes a more precise characterisation of classes without long zigzags.  A \emph{vertical alternation} is a permutation in which every odd-indexed entry lies above every even-indexed entry, or vice-versa. Some simple applications of the Erd\H{o}s-Szekeres Theorem shows that every sufficiently long vertical alternation contains a long \emph{parallel} or \emph{wedge} alternation -- see Figure~\ref{fig-alternations}.

\begin{figure}
{\centering
\begin{tikzpicture}[scale=.25]
	\plotpermgrid{2,9,4,7,6,12,1,8,5,11,3,10}
	\draw[thick] (.5,6.5) -- ++(12,0);	
\end{tikzpicture}
\qquad
\begin{tikzpicture}[scale=.25]
	\plotpermgrid{7,1,8,2,9,3,10,4,11,5,12,6}
	\draw[thick] (.5,6.5) -- ++(12,0);	
\end{tikzpicture}
\qquad
\begin{tikzpicture}[scale=.25]
	\plotpermgrid{7,6,8,5,9,4,10,3,11,2,12,1}
	\draw[thick] (.5,6.5) -- ++(12,0);	
\end{tikzpicture}\par}
\caption{From left to right: a vertical alternation, a parallel alternation, and a wedge alternation.}\label{fig-alternations}
\end{figure}
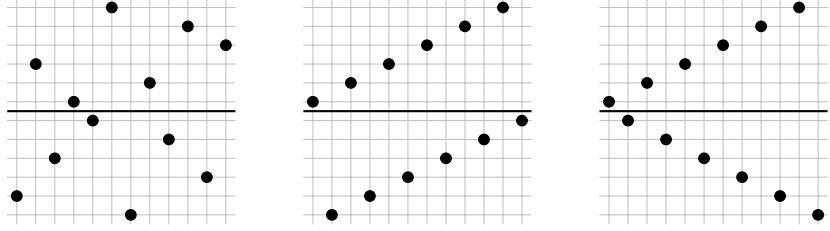

\begin{lemma}\label{lem:griddable}
The permutation class $\C$ contains only finitely many zigzags if and only if $\C$ is monotone griddable and does not contain arbitrarily long vertical alternations.
\end{lemma}

\begin{proof}
If $\C$ is not monotone griddable then it contains $\bigoplus21$ or $\bigominus12$ by Theorem~\ref{thm:hv}. In particular, for every $n\geq 1$, $\C$ contains either $2143\cdots (2n)(2n-1)$ or $(2n-1)(2n)\cdots 12$, both of which are zigzags. Similarly, if $\C$ contains arbitrarily long vertical alternations then for every $n\geq 1$ it contains a permutation of the form $a_1b_1a_2b_2\cdots a_nb_n$ where $\{a_1,\dots,a_n\}=\{n+1,\dots,2n\}$ and $\{b_1,\dots,b_n\}=\{1,\dots,n\}$, all of which are zigzags.

Conversely, Proposition~\ref{prop:w-class} shows that a class $\C$ with bounded length zigzags is contained in a monotone grid class comprising a single row, and this demonstrates both that $\C$ is monotone griddable and that it cannot contain arbitrarily long vertical alternations.
\end{proof}

We finish this section by recording an interesting consequence of the above theorem. The \emph{growth rate} of a permutation class $\C$ (or gridded permutation class $\C^\gridded$), if it exists, is $\lim_{n\to\infty}\sqrt[n]{|\C_n|}$, where $\C_n$ denotes the set of permutations in $\C$ of length $n$. The existence of the growth rate of a class in general depends upon whether the \emph{upper} and \emph{lower} growth rates coincide, that is, whether $\limsup_{n\to\infty}\sqrt[n]{|\C_n|}=\liminf_{n\to\infty}\sqrt[n]{|\C_n|}$.

\begin{cor}\label{cor:integral}
Let $\C$ be a class that contains only finitely many zigzags. Then $\gr(\C)$ exists and is integral.
\end{cor}

We need two auxiliary results. The first tells us that when a class is $\MMM$-griddable, then it suffices to consider the upper and lower growth rates of the gridded permutations.

\begin{prop}[{Vatter~\cite[Proposition~2.1]{vatter:small-permutati:}}]\label{prop:gr-grid}
For a matrix of permutation classes $\MMM$ and a class $\C\subseteq\Grid(\MMM)$, the upper or lower growth rate of $\C$ is equal, respectively, to the upper or lower growth rate of $\C^\gridded$.
\end{prop}

The second result is attributed to Albert in one of Vatter's seminal works regarding the growth rates of permutation classes.

\begin{prop}[{Attributed to Albert -- see Vatter~\cite[Proposition 7.4]{vatter:growth-rates-of:}}]\label{prop:gr-integral}
The growth rate of every subword-closed language exists and is integral.
\end{prop}

\begin{proof}[Proof of Corollary~\ref{cor:integral}]
By Proposition~\ref{prop:w-class}, we may suppose that $\C$ is contained in a monotone grid class $\Grid(\MMM)$ whose defining matrix comprises a single row of (say) $m$ cells. 

The set $\C^\gridded$ of all $\MMM$-gridded permutations in $\C$ is in bijection with a subword-closed language over an alphabet of size $m$ (see, for example, the description in Section~7 of Vatter~\cite{vatter:growth-rates-of:}), and in this bijection, the set of words corresponding to $\C$ is also subword-closed. By Proposition~\ref{prop:gr-integral}, the growth rate of $\C^\gridded$ exists and is integral, and thus by Proposition~\ref{prop:gr-grid} the same is true of the growth rate of $\C$. 
\end{proof}

%
%
%
%
%%%%%%%%%%
\section{Juxtapositions and lwqo}\label{sec:lwqo}

Since a class that contains only finitely many zigzags is $\MMM$-griddable for a monotone grid class formed of a single row, we now want to understand what happens when we juxtapose an arbitrary lwqo class $\C$ with such a grid class. The bulk of the remaining work lies in the next theorem, which establishes that lwqo is preserved whenever we juxtapose an lwqo class with $\av(21)$ or $\av(12)$.

\begin{thm}\label{thm:juxt-lwqo}
Let $\C$ be an arbitrary lwqo class, and let $\DDD$ be a monotone class. Then $\juxt{\C}{\DDD}$ is lwqo.
\end{thm}

\begin{proof}
By symmetry, we can assume that $\DDD =\av(21)$. Furthermore, it suffices to show that the gridded permutations, $\juxtbar{\C}{\DDD}$, are lwqo, since for any quasi-order $L$, the mapping $\Phi:\juxtbar{\C}{\DDD}\wr L\to\juxt{\C}{\DDD}\wr L$ that removes the gridline is an order-preserving surjection, and thus by Proposition~\ref{prop-order-preserving}, if $\juxtbar{\C}{\DDD}\wr L$ is wqo then so is $\juxt{\C}{\DDD}\wr L$. 

Let $(L,\leq_L)$ be an arbitrary wqo set of labels. By Higman's lemma, $(L^*,\preceq)$ is wqo. Furthermore, by Proposition~\ref{prop:product} the product $L\times L^*$ is also wqo, and thus $\C\wr (L\times L^*)$ is wqo since $\C$ is lwqo. Finally, another application of Proposition~\ref{prop:product} shows that $\C\wr (L\times L^*)\times L^*$ is wqo. 

A typical element of $\C\wr (L\times L^*)\times L^*$ has the form $\fP=((\pi,k_\pi),z_1\cdots z_q)$ where $\pi\in \C$ (of length $n$, say), $z_1,\dots,z_q\in L$, where $k_\pi: [n] \to L\times L^*$ is given by
\[ k_\pi(i) = (\ell(i),\lambda_{i1}\cdots\lambda_{in_i})\]
for all $i\in [n]$, in which $\ell:[n]\to L$, $\lambda_{ij}\in L$, and $n_i\geq 0$.

We now construct an order-preserving surjection $\Psi$ from $\C\wr (L\times L^*)\times L^*$ to $\juxtbar{\C}{\DDD}\wr L$. This mapping takes an object $\fP=((\pi,k_\pi),z_1\cdots z_q)$ and outputs an $L$-labelled permutation in $\juxtbar{\C}{\DDD}\wr L$ of length $n+\sum_{i=1}^n n_i+q$. Specifically, in $\Psi(\fP)$:
\begin{itemize}
\item There are $n$ points to the left of the gridline, order isomorphic to $\pi$. 
\item For $i\in[n]$, the $i$th point from the left is labelled by $\ell(i)$.
\item There are $\sum_{i=1}^n n_i+q$ points to the right of the gridline, forming an increasing sequence.
\item For $i\in[n]$, there are $n_i$ points to the right of the gridline that lie below the $i$th entry on the left, and above the next highest entry on the left (if this exists). These $n_i$ points are labelled $\lambda_{i1}$, $\dots$, $\lambda_{in_i}$ from bottom to top.
\item Above the highest entry on the left of the gridline, there are $q$ points to the right of the gridline, labelled $z_1,\dots,z_q$ from bottom to top.
\end{itemize}
See Figure~\ref{fig-psi}. The proof will be completed by showing that $\Psi$ is an order-preserving surjection.

\begin{figure}
\begin{tikzpicture}[scale=0.68]
\begin{scope}[scale=0.7,shift={(-.4,2)}]
	\plotpermgrid{3,5,1,6,2,4}
	\foreach \v/\seq [count=\x] in {3/\lambda_{11}\lambda_{12}\lambda_{13},
		5/\lambda_{21}\lambda_{22},1/\lambda_{31}\lambda_{32}\lambda_{33},
		6/\lambda_{41}\lambda_{42}\lambda_{43},2/\lambda_{51},
		4/\lambda_{61}\lambda_{62}}
	\node[empty] at ($(\v)+(-90:0.4)$) {\tiny$(\ell(\x),\seq)$};
	\node[empty] at (9,2.5) {$z_1z_2$};
\end{scope}
	
	\node[empty] at (7.75,3.5) {$\overset{\Psi}{\longrightarrow}$};
	\begin{scope}[shift={(8.5,0.5)}]
		\plotpermgrid{3,5,1,6,2,4}
		\draw[step=1cm,gray!50,very thin] (0.5,0.01) grid (13.99,6.99);
		\draw[thick] (6.5,0.01)--++(0,6.98);
		\foreach \v/\n [count=\x] in {3/4,5/3,1/4,6/4,2/2,4/3} {
			\node[empty] at ($(\v)+(-90:0.5)$) {$\ell(\x)$};
			\foreach \xx [count=\y] in {2,...,\n}
				\node[permpt,label={[label distance=-5pt]below right:\tiny$\lambda_{\v\y}$}] at (\x+6+\y/\n,\x-1+\y/\n) {};
		}
		\foreach \x in {1,2}
			\node[permpt,label={[label distance=-5pt]below right:\tiny$z_{\x}$}] at (13+\x/3,6+\x/3) {};
		
	\end{scope}
\end{tikzpicture}
\caption{The mapping $\Psi: \C\wr (L\times L^*)\times L^* \to \juxtbar{\C}{\DDD}\wr L$.}\label{fig-psi}
\end{figure}

First, any labelled gridded permutation in $\juxtbar{\C}{\DDD}\wr L$ comprises a set of points to the left of the gridline (that form a permutation from $\C$ with labels from $L$), interleaved by sequences of points to the right of the gridline (that form an increasing permutation, also with labels from $L$). With this in mind, for any specified element of $\juxtbar{\C}{\DDD}\wr L$ it is straightforward to identify a suitable preimage in $\C\wr (L\times L^*)\times L^*$, which shows that $\Psi$ is surjective.

Now consider $\fS=((\sigma,k_\sigma),w_1\cdots w_p)$ and $\fP=((\pi,k_\pi),z_1\cdots z_q)$ in $\C\wr (L\times L^*)\times L^*$, such that $\fS\leq \fP$.
 
Let $\sigma$ have length $m$ and $\pi$ length $n$. Since $\sigma \leq \pi$ as labelled permutations, there exists a subsequence $1\leq i_1<\cdots <i_m\leq n$ such that $\pi(i_1)\cdots\pi(i_m)$ is order isomorphic to $\sigma$, and $k_\sigma(j) \leq k_\pi(i_j)$ for all $j\in [m]$. If we write $k_\sigma(j) = (\ell_\sigma(j),\lambda_{j1}\cdots\lambda_{jm_{j}})$ and $k_\pi(i) = (\ell_\pi(i),\kappa_{i1}\cdots\kappa_{in_i})$, then $k_\sigma(j) \leq k_\pi(i_j)$ means that $\ell_\sigma(j)\leq_L\ell_\pi(i_j)$ and  $\lambda_{j1}\cdots\lambda_{jm_{j}}\preceq\kappa_{i_j1}\cdots\kappa_{i_jn_{i_j}}$ in generalised subword order. Finally, we also require $w_1\cdots w_p \preceq z_1\cdots z_q$.

To complete the proof, we show that $\Psi(\fS) \leq \Psi(\fP)$ as $L$-labelled gridded permutations.

The points to the left of the gridline in $\Psi(\fS)$ and $\Psi(\fP)$ form the $L$-labelled permutations $(\sigma,\ell_\sigma)$ and $(\pi,\ell_\pi)$, respectively. The subsequence $1\leq i_1<\cdots <i_m\leq n$ witnesses both that $\sigma\leq \pi$, and that $\ell_\sigma(j)\leq_L\ell_\pi(i_j)$, and hence $(\sigma,\ell_\sigma)\leq (\pi,\ell_\pi)$. We now consider the points to the right of the gridline. In $\Psi(\fS)$, for each $j\in[m]$ the points immediately below the entry on the left corresponding to $\sigma(j)$ form an increasing sequence of length $m_j$ labelled by $\lambda_{j1},\dots,\lambda_{jm_j}$. Similarly, in $\Psi(\fP)$, the points immediately below the entry corresponding to $\pi(i_j)$ form an increasing sequence of length $n_{i_j}$ labelled by $\kappa_{i_j1},\dots,\kappa_{i_jn_{i_j}}$. Since $\lambda_{j1}\cdots\lambda_{jm_{j}}\preceq \kappa_{i_j1}\cdots\kappa_{i_jn_{i_j}}$, we can embed these $m_j$ labelled points of $\Psi(\fS)$ in the $n_{i_j}$ labelled points of $\Psi(\fP)$.

Finally, in $\Psi(\fS)$, there are $p$ labelled entries to the right of the gridline that lie above all entries to the left of the grid line. Since $w_1\cdots w_p \preceq z_1\cdots z_q$, these $p$ entries can be embedded in the $q$ entries of $\Psi(\fP)$ in the top-right. We have now embedded every labelled entry of $\Psi(\fS)$ in $\Psi(\fP)$, and the proof is complete.
\end{proof}

Our approach to resolve one direction of Theorem~\ref{thm:main} will be to apply the preceding theorem iteratively. For the other direction, we appeal to pre-existing antichain constructions, which are succinctly summarised by the following theorem.

The \emph{cell graph} of a matrix $\MMM$ is the graph whose vertices are $\{(i,j):M_{ij}\neq \varnothing\}$ (corresponding to the non-empty cells of $\MMM$), and $(i,j)\sim(k,\ell)$ if and only if $i=k$ or $j=\ell$, and there are no non-empty cells between these $M_{ij}$ and $M_{k\ell}$ in their common row or column.

\begin{thm}[{See Brignall~\cite[Theorem 1.1]{brignall:pwo-grid-classes:}}]\label{thm:brignall}
	Let $\MMM$ be a gridding matrix where every non-empty cell is an infinite permutation class. Then $\Grid(\MMM)$ is not well-quasi-ordered whenever the cell graph of $\MMM$ has a cycle, or a component containing two or more cells that are not monotone griddable.
\end{thm}

Note that the `cyclic' case of the above theorem is originally due to Murphy and Vatter~\cite{murphy:profile-classes:}.

\begin{proof}[Proof of Theorem~\ref{thm:main}]
If one of $\C$ or $\DDD$ is not lwqo, then clearly neither is $\juxt{\C}{\DDD}$ since it contains both $\C$ and $\DDD$ as subclasses. So now suppose both $\C$ and $\DDD$ are lwqo, but contain arbitrarily long zigzags. By Lemma~\ref{lem:griddable} each of $\C$ and $\DDD$ either is not monotone griddable, or contains arbitrarily long vertical alternations (or both). 

If neither $\C$ nor $\DDD$ is monotone griddable, then $\juxt{\C}{\DDD}$ is not wqo (and thus also not lwqo) by Theorem~\ref{thm:brignall}. (See Figure~\ref{fig-antichains} (left) for a typical antichain element in this case.)

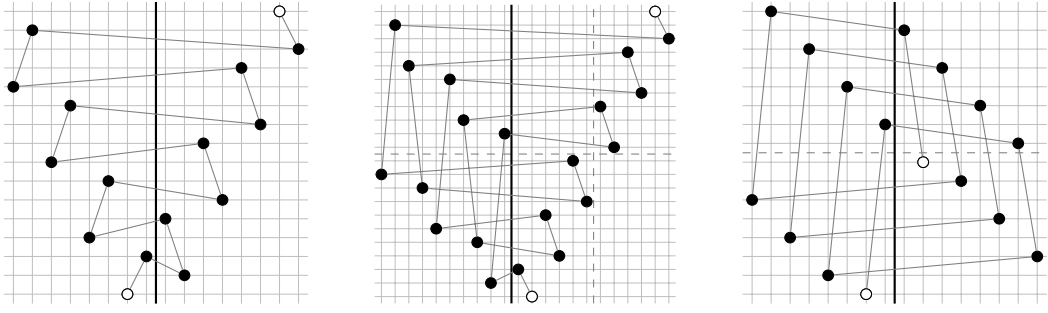
\begin{figure}
{\centering
\begin{tikzpicture}[scale=0.25]
	\plotpermgrid{12,15,8,11,4,7,1,3,5,2,9,6,13,10,16,14}
	\draw[thick] (8.5,0.5) -- ++ (0,16);
	\node[permpt,fill=white] at (1) {};
	\node[permpt,fill=white] at (16) {};
	\draw[thin,black!50] (1)--(3)--(2)--(5)--(4)--(7)--(6)--(9)--(8)--(11)--(10)--(13)--(12)--(15)--(14)--(16);
\end{tikzpicture}
\qquad
\begin{tikzpicture}[scale=0.18]
	\plotpermgrid{10,21,18,9,6,17,14,5,2,13,3,1,7,4,11,8,15,12,19,16,22,20}
	\draw[dashed,black!50] (16.5,0.5) -- ++ (0,22)
					 (0.5,11.5) -- ++ (22,0);
	\draw[thick] (10.5,0.5) -- ++ (0,22);
	\node[permpt,fill=white] at (1) {};
	\node[permpt,fill=white] at (22) {};
	\draw[thin,black!50] (1)--(3)--(2)--(13)--(12)--(15)--(14)--(5)--(4)--(7)--(6)--(17)--(16)--(19)--(18)--(9)--(8)--(11)--(10)--(21)--(20)--(22);
\end{tikzpicture}
\qquad\begin{tikzpicture}[scale=0.25]
	\plotpermgrid{6,16,4,14,2,12,1,10,15,8,13,7,11,5,9,3}
	\draw[dashed,black!50] (0.5,8.5) -- ++ (16,0);
	\draw[thick] (8.5,0.5) -- ++ (0,16);
	\node[permpt,fill=white] at (1) {};
	\node[permpt,fill=white] at (8) {};
	\draw[thin,black!50] (1)--(10)--(9)--(3)--(2)--(12)--(11)--(5)--(4)--(14)--(13)--(7)--(6)--(16)--(15)--(8);
\end{tikzpicture}
\par}
\caption{Typical labelled antichain elements arising in juxtaposition classes. Here, we may take $L=\{\bullet,\circ\}$ to be an antichain of size 2.}\label{fig-antichains}
\end{figure}

Now suppose, without loss of generality, that $\C$ is monotone griddable but contains long vertical alternations, and $\DDD$ is not monotone griddable. By Theorem~\ref{thm:hv}, the class $\DDD$ contains $\bigoplus21$ or $\bigominus12$. Consequently, $\juxt{\C}{\DDD}$ contains $\Grid(\MMM)$ for a matrix $\MMM$ of the following form:
\[
\MMM = \begin{bmatrix}\EEE_1&&\bigoplus21\\\EEE_2&\bigoplus 21\end{bmatrix} \quad\text{or}\quad \begin{bmatrix}\EEE_1&\bigominus12\\\EEE_2&&\bigominus 12\end{bmatrix}
\]
where $\EEE_1$ and $\EEE_2$ are each either $\av(21)$ or $\av(12)$. In any case, the cell graph of $\MMM$ comprises a component containing two cells that are not monotone griddable (again by Theorem~\ref{thm:hv}), and hence $\Grid(\MMM)$ is not wqo by Theorem~\ref{thm:brignall}. (See Figure~\ref{fig-antichains} (middle) for a typical antichain element in this case.)

Finally for this direction, suppose that both $\C$ and $\DDD$ are monotone griddable, but both contain arbitrarily long vertical alternations. In this case, $\juxt{\C}{\DDD}$ contains $\Grid(\MMM)$ for a matrix $\MMM$ of the following form
\[
\MMM = \begin{bmatrix}\EEE_1&\EEE_2\\\EEE_3&\EEE_4\end{bmatrix}
\]
where $\EEE_1$, $\EEE_2$, $\EEE_3$ and $\EEE_4$ are each either $\av(21)$ or $\av(12)$. In any case, the cell graph of $\MMM$ comprises a component that is a cycle, so $\Grid(\MMM)$ is once again not wqo by Theorem~\ref{thm:brignall}, and hence neither is $\juxt{\C}{\DDD}$. (See Figure~\ref{fig-antichains} (right) for a typical antichain element in this case.)

For the other direction, suppose (without loss of generality) that $\C$ is lwqo, and $\DDD$ contains only bounded length zigzags. By Proposition~\ref{prop:w-class}, there exists a single-row monotone grid class $\EEE$ such that $\DDD\subseteq \EEE$. We claim that $\juxt{\C}{\EEE}$ is lwqo.

Write $\EEE=\Grid(\MMM)$ where $\MMM=\begin{bmatrix}\EEE_1&\EEE_2&\cdots&\EEE_k\end{bmatrix}$ for classes $\EEE_i$ each equal to $\av(21)$ or $\av(12)$ ($1\leq i\leq k$). Let $\C_0=\C$, and for $1\leq i\leq k$ set
\[\C_i = \Grid(\begin{bmatrix}\C&\EEE_1 &\cdots &\EEE_i\end{bmatrix}).\]
Now $C_0=\C$ is lwqo, and it follows by induction and Theorem~\ref{thm:juxt-lwqo} that $\C_i = \juxt{\C_{i-1}}{\EEE_i}$ is lwqo for each $i=1,\dots,k$. In particular $\C_k = \juxt{\C}{\EEE}$ is lwqo. The result now follows since $\juxt{\C}{\DDD}\subseteq \juxt{\C}{\EEE}$. 
\end{proof}

%
%
%
%
%
%
%
%
%%%%%%%%%%%%%%%%%%%%%%
\section{Concluding remarks}\label{sec:conclusion}

The methods and ideas in this note can almost certainly be adapted to a characterisation of lwqo in grid classes, although it would likely be technically and notationally awkward to do so. 

A more interesting future direction is to consider lwqo in \emph{subclasses} of these grid classes. For example, while the juxtaposition of $\bigominus 12$ with $\bigoplus 21$ contains the infinite antichain comprising elements of the form shown on the left of Figure~\ref{fig-antichains}, there exist subclasses of this juxtaposition that are lwqo. Individual cases such as this are relatively easy to characterise, but a general answer seems further out of reach.

Can a similar characterisation can be achieved for (unlabelled) wqo? Although the antichain elements depicted in Figure~\ref{fig-antichains} use two labels, the proof of Theorem~\ref{thm:main} in fact uses only unlabelled antichains, so aspects of this question already have an answer. However, if $\C$ is a wqo-but-not-lwqo class, then it is sometimes possible to break wqo by juxtaposing $\C$ with the class containing just the singleton permutation, while in other cases, $\C$ must be juxtaposed with two entries. In general, we cannot hope to make progress on this question without a significantly deeper understanding of wqo in permutation classes.

%Here, we study the lwqo properties some \emph{subclasses} of juxtapositions. We let $\EEE\subseteq \juxt{\C}{\DDD}$. Our chief interest is in situations where $\EEE$ is expressed as a \emph{minimal} juxtaposition $\juxt{\C}{\DDD}$, in the sense that if $\EEE\subseteq \juxt{\C'}{\DDD'}$ such that both $\C'\subseteq\C$ and $\DDD'\subseteq\DDD$, then in fact $\C'=\C$ and $\DDD'=\DDD$. 

%Note that the above does not preclude there being other expressions of $\EEE$ as a juxtaposition of two classes, where, for example, we take a proper subclass of $\C$, and a proper superclass of $\DDD$. One can envisage `moving' the vertical dividing line to the left or right in order to change the contents in each class. Here however, we will not concern ourselves with such considerations: we are given the class $\EEE$ and a minimal juxtaposition $\juxt{\C}{\DDD}$ that contains it.

%
%
%
%
%
%
%
%
%%%%%%%%%%%%%%%%%%%%%%
%\section{Concluding remarks}

% ------------------------------------------------------------------------------------------------
\bibliographystyle{plain}
%{\footnotesize\bibliography{bib}}
\bibliography{refs}

\def\cprime{$'$}
\begin{thebibliography}{10}

\bibitem{oeis}
The {O}n-line {E}ncyclopedia of {I}nteger {S}equences.
\newblock {P}ublished electronically at
  \href{http://oeis.org/}{http://oeis.org/}.

\bibitem{atkinson:restricted-perm:}
M.~D. Atkinson.
\newblock Restricted permutations.
\newblock {\em Discrete Math.}, 195(1-3):27--38, 1999.

\bibitem{atkinson:partially-well-:}
M.~D. Atkinson, M.~M. Murphy, and N.~Ru{\v{s}}kuc.
\newblock Partially well-ordered closed sets of permutations.
\newblock {\em Order}, 19(2):101--113, 2002.

\bibitem{bevan2015defs}
David Bevan.
\newblock Permutation patterns: basic definitions and notation.
\newblock arXiv:1506.06673, 2015.

\bibitem{billey:permutations-with:}
Sara Billey, Krzysztof Burdzy, and Bruce~E. Sagan.
\newblock Permutations with given peak set.
\newblock {\em J. Integer Seq.}, 16(6):Article 13.6.1, 18, 2013.

\bibitem{brignall:pwo-grid-classes:}
Robert Brignall.
\newblock Grid classes and partial well order.
\newblock {\em J. Combin. Theory Ser. A}, 119(1):99--116, 2012.

\bibitem{bs:context-free}
Robert Brignall and Jakub Slia\v{c}an.
\newblock Combinatorial specifications for juxtapositions of permutation
  classes.
\newblock {\em Electron. J. Combin.}, 26(4):Paper No. 4.4, 24, 2019.

\bibitem{bv:lwqo-for-pp:}
Robert Brignall and Vincent Vatter.
\newblock Labelled well-quasi-order for permutation classes.
\newblock {\em Comb. Theory}, 2(3):Paper No. 14, 54, 2022.

\bibitem{higman:ordering-by-div:}
Graham Higman.
\newblock Ordering by divisibility in abstract algebras.
\newblock {\em Proc. London Math. Soc. (3)}, 2:326--336, 1952.

\bibitem{Huczynska2006}
Sophie Huczynska and Vincent Vatter.
\newblock Grid classes and the {F}ibonacci dichotomy for restricted
  permutations.
\newblock {\em Electron. J. Combin.}, 13:Research paper 54, 14 pp.
  (electronic), 2006.

\bibitem{murphy:profile-classes:}
Maximillian~M. Murphy and Vincent Vatter.
\newblock Profile classes and partial well-order for permutations.
\newblock {\em Electron. J. Combin.}, 9(2):Research paper 17, 30 pp.
  (electronic), 2003.

\bibitem{nyman:the-peak:}
Kathryn~L. Nyman.
\newblock The peak algebra of the symmetric group.
\newblock {\em J. Algebraic Combin.}, 17(3):309--322, 2003.

\bibitem{pouzet:un-bel-ordre-da:}
Maurice Pouzet.
\newblock Un bel ordre d'abritement et ses rapports avec les bornes d'une
  multirelation.
\newblock {\em C. R. Acad. Sci. Paris S\'er. A-B}, 274:A1677--A1680, 1972.

\bibitem{pratt:computing-permu:}
Vaughan~R. Pratt.
\newblock Computing permutations with double-ended queues, parallel stacks and
  parallel queues.
\newblock In {\em STOC '73: Proceedings of the Fifth Annual ACM Symposium on
  the Theory of Computing}, pages 268--277, New York, NY, USA, 1973. ACM Press.

\bibitem{stanley:survey-alternating:}
Richard~P. Stanley.
\newblock A survey of alternating permutations.
\newblock In {\em Combinatorics and graphs}, volume 531 of {\em Contemp.
  Math.}, pages 165--196. Amer. Math. Soc., Providence, RI, 2010.

\bibitem{tarjan:sorting-using-n:}
Robert Tarjan.
\newblock Sorting using networks of queues and stacks.
\newblock {\em J. Assoc. Comput. Mach.}, 19:341--346, 1972.

\bibitem{vatter:small-permutati:}
Vincent Vatter.
\newblock Small permutation classes.
\newblock {\em Proc. Lond. Math. Soc. (3)}, 103:879--921, 2011.

\bibitem{vatter:growth-rates-of:}
Vincent Vatter.
\newblock Growth rates of permutation classes: from countable to uncountable.
\newblock {\em Proc. Lond. Math. Soc. (3)}, 119(4):960--997, 2019.

\end{thebibliography}

\end{document}